\documentclass{amsart}

\usepackage{amsmath,amssymb,amsthm}
\usepackage{a4wide}
\usepackage{hyperref}
\usepackage{multicol}

\usepackage{graphicx}
\usepackage{xypic}
\entrymodifiers={+!!<0pt,\fontdimen22\textfont2>}
\usepackage[all]{xy}

\newtheoremstyle{myremark} 
    {7pt}                    
    {7pt}                    
    {}  	                 
    {}                           
    {\bf}       	         
    {.}                          
    {.5em}                       
    {}  

\theoremstyle{plain}
\newtheorem{lemma}{Lemma}
\newtheorem{theorem}[lemma]{Theorem}

\newtheorem{proposition}[lemma]{Proposition}

\theoremstyle{myremark}

\newcommand{\nicef}{\mathcal{F}}
\newcommand{\nicem}{\mathcal{M}}

\newcommand{\longversion}[1]{}

\begin{document}
\title{Face numbers of down-sets}

\author[Micha{\l} Adamaszek]{Micha{\l} Adamaszek}
\address{Institute of Mathematics, University of Bremen
  \newline Bibliothekstr. 1, 28359 Bremen, Germany}
\email{aszek@mimuw.edu.pl}


\begin{abstract}
We compare various viewpoints on down-sets (simplicial complexes), 
illustrating how the combinatorial inclusion-exclusion principle may serve as
an alternative to more advanced methods of studying their face numbers.

\end{abstract}
\maketitle

\noindent
For any family $\Delta\subseteq 2^{[n]}$ of subsets of some ground set $[n]=\{1,\ldots,n\}$ we define a two-variable generating function
$$H_\Delta(x,y)=\sum_{\sigma\in\Delta} x^{|\sigma|}y^{n-|\sigma|}$$
and its two specializations, the $f$- and $K$-polynomial:
$$f_\Delta(t)=H_\Delta(t,1) \quad\textrm{and}\quad K_\Delta(t)=H_\Delta(t,1-t).$$

The family $\Delta$ is called a \emph{down-set} if it is closed under taking subsets; that is $\sigma\in\Delta$ and $\tau\subseteq \sigma$ imply $\tau\in\Delta$. For a down-set $\Delta$, let $\nicef=\{F_1,\ldots,F_m\}$ be the family of \emph{maximal elements} in $\Delta$ and let $\nicem=\{M_1,\ldots,M_k\}$ be the family of minimal elements of $2^{[n]}\setminus \Delta$, which will also be called the \emph{blockers} of $\Delta$. 

Clearly every object of the triple $\Delta$, $\nicef$, $\nicem$ determines the other two, 
but the representation via $\nicef$ or $\nicem$ is usually much smaller than the listing of all of $\Delta$. 
We will start with a brief survey of down-sets in various branches of
mathematics. Different points of view highlight different aspects of the
relationship between $\nicef$, $\nicem$, $f_\Delta(t)$ and $K_\Delta(t)$.
We will then give a formula for $H_\Delta(x,y)$ in terms of the combinatorics
of $\nicef$ and $\nicem$ and show its interpretations in the different 
approaches to down-sets.

The ways of thinking of down-sets we wish to consider are the following.
\begin{itemize}
\item In geometry, a down-set is usually called an \emph{abstract simplicial complex}. The elements of $\Delta$ are its \emph{faces}, enumerated by the face polynomial $f_\Delta(t)$ which is a well-studied combinatorial invariant (see \cite{SS}). The elements of $\nicef$ are the highest-dimensional faces (or \emph{facets}) of $\Delta$. By taking a simplex spanned on $\ell$ vertices for each $\ell$-element face in $\Delta$ and gluing them together, we obtain the geometric realization of $\Delta$, as in the figure.
\begin{figure}[ht!]
\begin{center}
\includegraphics{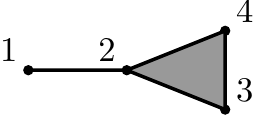}
\caption{The geometric realization of the down-set $\Delta=\{\emptyset,1,2,3,4,12,23,24,34,234\}\subseteq 2^{[4]}$ with maximal elements $\nicef=\{12,234\}$ and blockers $\nicem=\{13,14\}$}
\end{center}
\end{figure}

\item In commutative algebra, every element $M\in\nicem$ can be identified with a square-free monomial $\mathbf{x}_M=\prod_{i\in M} x_i$ in the polynomial ring $R=\mathbb{F}[x_1,\ldots,x_n]$. If $I=(\mathbf{x}_{M_1},\ldots,\mathbf{x}_{M_k})$, then $K_\Delta(t)/(1-t)^n$ is the Hilbert series of the $R$-module $R/I$, with the latter usually referred to as the Stanley-Reisner ring of $\Delta$. As one would expect, the correspondence between simplicial complexes and their Stanley-Reisner rings is the source of a fruitful interaction between combinatorial geometry and commutative algebra.

\item In complexity theory, $\nicem$ determines a monotone Boolean function in disjunctive normal form
$$\varphi_\nicem(x_1,\ldots,x_n)=\bigvee_{M\in\nicem}\bigwedge_{i\in M} x_i$$
with variables $x_1,\ldots, x_n$. Then $\Delta$ is the set of non-satisfying assignments for $\varphi_\nicem$. The problem of expressing the formula $\overline{\varphi_\nicem}(\overline{x_1},\ldots,\overline{x_n})$ again in disjunctive normal form, known as \emph{monotone dualization} or \emph{hypergraph transversal}, is equivalent to computing $\nicef$ from $\nicem$. The decision version of this problem is ``Given $\nicem$ and $\nicef$, do they represent the same $\Delta$?''. The question has received a lot of attention due to its applicability in various areas of theoretical computer science, and it is open whether it can be answered in time polynomial with respect to the size of the combined description of $\nicem$ and $\nicef$. The best known algorithm requires quasi-polynomial time \cite{D}.
\end{itemize}

The reader is welcome to consult \cite{SM,EMG} for a more thorough treatment of the above topics. For example, \cite{SM} provides an exposition based on the ideas of combinatorial commutative algebra. Here is our main result, illustrating how the elementary inclusion-exclusion principle may serve as an alternative to the homological algebra of monomial ideals.

\begin{theorem}
\label{mainthm}
For a down-set $\Delta\subseteq 2^{[n]}$ with maximal elements $\nicef$ and blockers $\nicem$ we have
\begin{eqnarray*}
H_\Delta(x,y)&=& \sum_{\emptyset\neq S\subseteq\nicef}(-1)^{|S|+1}(x+y)^{|\bigcap S|}y^{n-|\bigcap S|}\\
&=&\sum_{T\subseteq\nicem}(-1)^{|T|}(x+y)^{n-|\bigcup T|}x^{|\bigcup T|}.
\end{eqnarray*}
\end{theorem}
\begin{proof}
The coefficient of $x^\ell y^{n-\ell}$ in $H_\Delta(x,y)$ is the number of sets of cardinality $\ell$ in $\Delta$. It is given by either of the formulae
$$\sum_{\emptyset\neq S\subseteq\nicef}(-1)^{|S|+1}{|\bigcap S|\choose \ell}\quad\textrm{or}\quad\sum_{T\subseteq\nicem}(-1)^{|T|}{n-|\bigcup T|\choose \ell-|\bigcup T|}.$$
Both expressions are instances of the inclusion-exclusion principle, except that instead of including-excluding elements, as it is usually done, we include-exclude entire $\ell$-subsets of sets in $\nicef$ or $\ell$-supersets of sets in $\nicem$. From this the formulae for $H_\Delta(x,y)$ immediately follow.
\end{proof}

Here are some applications.
\begin{itemize}
\item[1)] We can express the $K$-polynomial of $\Delta$ in terms of the blockers $\nicem$:
\begin{equation}
\label{eq:kk}
K_\Delta(t)=H_\Delta(t,1-t)=\sum_{T\subseteq\nicem}(-1)^{|T|}t^{|\bigcup T|}.
\end{equation}
This is exactly the main result of \cite{G}, with special cases appearing also in earlier work. In fact the formula $H_\Delta(x,y)=(x+y)^nK_\Delta(x/(x+y))$ implies that \eqref{eq:kk} is equivalent to Theorem~\ref{mainthm}.
\item[2)] The reduced Euler characteristic of $\Delta$ is $\widetilde{\chi}(\Delta)=f_\Delta(-1)=H_\Delta(-1,1)$. In a similar fashion we get
$$\widetilde{\chi}(\Delta)=H_\Delta(-1,1)=\sum_{T\subseteq\nicem}(-1)^{|T|}0^{n-|\bigcup T|}(-1)^{|\bigcup T|}=(-1)^n\sum_{\substack{T\subseteq\nicem\\\bigcup T=[n]}}(-1)^{|T|},$$
as shown in \cite[Thm.4.2]{MT}.
\item[3)] The dual version of the above is
\begin{eqnarray*}
\widetilde{\chi}(\Delta)=\sum_{\emptyset\neq S\subseteq\nicef}(-1)^{|S|+1}0^{|\bigcap S|}&=&\sum_{\substack{S\subseteq\nicef\\\bigcap S=\emptyset}}(-1)^{|S|+1}
=\sum_{\substack{S\subseteq\nicef\\\bigcap S\neq\emptyset}}(-1)^{|S|},
\end{eqnarray*}
which follows also from the fact, familiar to topologists, that a simplicial complex $\Delta$ is homotopy equivalent to the nerve of the family $\nicef$, hence their Euler characteristics are equal.
\item[4)] In the language of the $f$-polynomial Theorem~\ref{mainthm} reads
$$f_\Delta(t)=H_\Delta(t,1)=\sum_{\emptyset\neq S\subseteq\nicef}(-1)^{|S|+1}(t+1)^{|\bigcap S|}.$$
\end{itemize}

It is worth mentioning that the apparent symmetry of the two formulae in Theorem~\ref{mainthm} is a manifestation of Alexander duality, which prevails in all examples of down-sets we mentioned earlier. For a down-set $\Delta\subseteq 2^{[n]}$ define
$$\Delta^*=\{[n]\setminus\sigma~:~\sigma\in 2^{[n]}\setminus\Delta\}.$$
Then $\Delta^*$ is again a down-set which we call the \emph{Alexander dual} of $\Delta$. The maximal elements of $\Delta^*$ are $F_i^*=[n]\setminus M_i$ for the blockers $M_1,\ldots,M_k$ of $\Delta$, while the blockers of $\Delta^*$ are $M_j^*=[n]\setminus F_j$ for the maximal elements $F_1,\ldots,F_m$ of $\Delta$. We will denote those new families by $\nicef^*=\{F_1^*,\ldots,F_k^*\}$ and $\nicem^*=\{M_1^*,\ldots,M_m^*\}$. Note that $(\Delta^*)^*=\Delta$ and $\overline{\varphi_\nicem}(\overline{x_1},\ldots,\overline{x_n})=\varphi_{\nicem^*}(x_1,\ldots,x_n)$.

\begin{proposition}
\label{prop}
For a down-set $\Delta\subseteq 2^{[n]}$ we have 
$$H_\Delta(x,y)+H_{\Delta^*}(y,x)=(x+y)^n.$$
\end{proposition}
\begin{proof}
We will give a proof which uses Theorem~\ref{mainthm}, but the reader is encouraged to find a proof based only on the definition of $H_\Delta(x,y)$ and then check that under Proposition~\ref{prop} each of the two parts of Theorem~\ref{mainthm} implies the other. We obtain
\begin{align*}
&H_\Delta(x,y)+ H_{\Delta^*}(y,x)=\\
&=\sum_{T\subseteq\nicem}(-1)^{|T|}(x+y)^{n-|\bigcup T|}x^{|\bigcup T|}+\sum_{\emptyset\neq S^*\subseteq\nicef^*}(-1)^{|S^*|+1}(x+y)^{|\bigcap S^*|}x^{n-|\bigcap S^*|}\\
&=\sum_{T\subseteq\nicem}(-1)^{|T|}(x+y)^{n-|\bigcup T|}x^{|\bigcup T|}+\sum_{\emptyset\neq T\subseteq\nicem}(-1)^{|T|+1}(x+y)^{n-|\bigcup T|}x^{|\bigcup T|}.
\end{align*}
The only contribution which is not cancelled out is $(x+y)^n$ arising from $T=\emptyset$.
\end{proof}

\medskip
If we only perform the first approximation to the inclusion-exclusion formula, that is take the union bound, after some manipulation we derive for $x,y\geq 0$ the inequality
$$H_\Delta(x,y)\leq \sum_{F\in\nicef} (x+y)^{|F|}y^{n-|F|} = \sum_{M^*\in\nicem^*} (x+y)^{n-|M^*|}y^{|M^*|}.$$
Taking $x=y=\frac12$ in Proposition~\ref{prop} and using the above we obtain
\begin{equation}
\label{eq}
\sum_{M^*\in \nicem^*}2^{-|M^*|}+\sum_{M\in \nicem}2^{-|M|}\geq 1,
\end{equation}
which is well known in the theory of the dualization of disjunctive normal forms \cite[Lemma 1]{D} as a tool for estimating the joint size of a formula and its dual. To see how much \eqref{eq} deviates from equality one takes $x=y=\frac12$ in Proposition~\ref{prop}:
$$\sum_{T^*\subseteq\nicem^*}(-1)^{|T^*|}2^{-|\bigcup T^*|}+\sum_{T\subseteq\nicem}(-1)^{|T|}2^{-|\bigcup T|}=1.$$

\subsection*{Acknowledgment.} The author thanks the referees for helpful remarks concerning the presentation. This work was supported by the DFG grant FE 1058/1-1.


\end{document}